\documentclass{amsart}
\usepackage[english]{babel}
\usepackage[utf8]{inputenc}
\usepackage{tikz, amsthm, amsmath, amsfonts, tikz-cd, verbatim, enumitem, mathtools, adjustbox, amssymb, scalerel}
\usepackage{xcolor}
\usetikzlibrary{patterns}

\newtheorem{theorem}{Theorem}

\newtheorem{prop}[theorem]{Proposition}
\newtheorem{lem}[theorem]{Lemma}
\newtheorem{note}[theorem]{Note}
\newtheorem{defn}[theorem]{Definition}
\newtheorem{exmp}[theorem]{Example}
\newtheorem{cor}[theorem]{Corollary}

\newcommand{\RR}{\mathbb{R}}

\usepackage{biblatex}
\title{Stability for Layer Points}
\author{Katharine L. M. Adamyk}
\address{Western University, Department of Mathematics, Middlesex College,
London, Ontario, Canada  N6A 5B7}
\email{kadamyk@uwo.ca}

\begin{document}

\maketitle

\section*{Abstract}
In the first half this paper, we generalize the theory of layer points \cite{Layers} to the more general context of $\vec{v}$-hierarchical clusterings \cite{RS}.
Layer points provide a compressed description of a hierarchical clustering by recording only the points where a cluster changes.  For multi-parameter hierarchical clusterings we consider both a global notion of layer points and layer points in the direction of a single parameter.  An interleaving of hierarchical clusterings of the same set induces an interleaving of global layer points.  In the particular, we consider cases where a hierarchical clustering of a finite metric space, $Y$, is interleaved with a hierarchical clustering of some sample $X \subseteq Y$.

In the second half, we focus on the hierarchical clustering $\pi_0 L_{-,k}(Y)$ for some finite metric space $Y$.  When $X \subseteq Y$ satisfies certain conditions guaranteeing $X$ is well dispersed in $Y$ and the points of $Y$ are dense around $X$, there is an interleaving of layer points for $\pi_0 L_{-,k}(Y)$ and a truncated version of $L_{-,0}(X) = V_{-}(X)$.  Under stronger conditions, this interleaving defines a retract from the layer points for $\pi_0 L_{-,k}(Y)$ to the layer points for $\pi_0 L_{-,0}(X)$.

\section{Introduction}

At its most basic level, a clustering of a data set is any collection of its disjoint subsets.  Of course, in practice, these subsets should indicate some relationship between data points and there are a wide variety of methods that attempt to form clusters in a meaningful way. 
If a clustering method depends on a choice of parameter, we can consider how the clustering changes as the parameter varies.  Heuristically, in a (covariant) hierarchical clustering, clusters appear or merge together as the parameter increases.  (Whereas, in a contravariant hierarchical clustering, clusters break apart or disappear as the parameter increases.)  

Rather than a single parameter, a hierachical clustering may be indexed instead by the elements of any partially ordered set.  In particular,   we will utilize the $\vec{v}$-hierarchial clusterings of Rolle and Scoccola \cite{RS} as a framework for multi-parameter hierarchical clusterings.   
Consider $H$, a $\vec{v}$-hierarchical clustering of a set $Z$ for some $\vec{v} \in \{-1,1\}^n$.  Such a hierarchical clustering is indexed by the elements of $R^{\vec{v}}$, which is the set $\RR^{n}$ with a partial order that depends on $\vec{v}$ (see Section~\ref{subsec:background}).  
The set \[\Gamma(H) = \{ (\vec{s}, S) | \vec{s} \in R^{\vec{v}}, S \in H(\vec{s})\},\]
has a partial order
where $(\vec{s}, S) \leq (\vec{t}, T)$ if $\vec{s} \leq \vec{t}$ and $S \subseteq T$.  This is a sub-poset of $R^{\vec{v}} \times \mathcal{P}(Z)$ where $\mathcal{P}(Z)$ is the power set of $Z$ with the partial order given by inclusion. 

Consider the equivalence relation on $\Gamma(H)$ generated by setting $(\vec{t}, T) \sim (\vec{s}, S)$ whenever $\vec{t} \leq \vec{s}$ and $S=T$.  The equivalence classes of this equivalence relation are called the layers of $\Gamma(H)$. 
Alternately, layers can be viewed as the path components of a graph whose vertices are the elements of $\Gamma(H)$ with edges $(\vec{s}, S) \to (\vec{t}, S)$ for all $\vec{s} \leq \vec{t}$ with $S \in H(\vec{s}) \cap H(\vec{t})$.  
This graph is the component graph defined in \cite{Layers}.  

Layer points are distinguished points on the boundary of a layer, and we consider two types of them. 
A global layer point for a $\vec{v}$-hierarchical clustering is some $(\vec{s}, S)$ in $\Gamma(H)$ where $S$ does not appear in $H(\vec{t})$ for any $\vec{t} < \vec{s}$.  
An $i^{th}$-parameter layer point of a multiparameter hierarchical clustering, $H$, is a global layer point of a single-parameter slice of $H$.  
In both cases, a layer point marks a first occurrence of a cluster, keeping in mind that when considering multiple parameters the ``first'' occurrence may not be unique.  In order to guarantee the existence of such points, we restrict our attention to $\vec{v}$-hierarchical clusterings that are bounded below (see Definition~\ref{defn:bbdbelow}). 

Taking maximal lower bounds within layers defines a (not necessarily unique) map from $\Gamma(H)$ to layer points.  This leads to our first stability result. 
The layer points of single parameter hierarchical clusterings are stable in the sense that an interleaving of hierarchical clusterings induces a homotopy interleaving of layer points (Theorem~\ref{thm:stability1}).  In particular, we will be interested in these interleavings of layer points for approximations of a hierarchical clustering by a hierarchical clustering of a subsample of points. 

A standard example of a 2-parameter hierarchical clustering is given by taking the path components of the Lesnick, or degree-Rips, complex associated to a subset of a metric space. The Lesnick complexes, $L_{s,k}(Z)$ filter the Vietoris-Rips complex, $V_s (X)$. (See Section~\ref{sec:layer_stab} for a brief introduction to these complexes.)  This is the context in which layer points were originally introduced by Jardine in \cite{Layers}.  We consider the single parameter slices of this clustering where the density parameter is constant.

Let $X$ be a sample (subset) of a finite metric space $Y$.  We define the hierarchical clustering $L_0 X[c]$, a truncated version of $\pi_0 L_{-, 0}(X)=\pi_0 V_{-}(X)$, as an approximation for $L_k Y = \pi_0 L_{-,k}(Y)$, the $2^{nd}$-parameter slice of $\pi_0 L_{-,-}(Y)$.  This approximation is stable---if $X$ and $Y$ satisfy an appropriate notion of closeness, then there is an interleaving of $L_0 X[c]$ and $L_k Y$ (Lemma~\ref{lem:posets}). 

To define what it means for $X$ and $Y$ to be close, we consider two directional notions of distance between $X$ and $Y$.  The first is the directional Hausdorff distance, $h(Y, X)$.
The second is a directional distance adjusted for density, $N_k(Y, X)$.  
The value of $N_k(X, Y)$ gives a threshold for the parameter at which all points in $X$ have $k$ neighbours in $Y$. The stability results presented in this paper apply to sets $X \subseteq Y$ where $h(Y, X)$ and $N_k(X,Y)$ are small.  

Let $s_0 < s_1 < \ldots < s_M$ be the phase change numbers of $X$, which are the distinct distances between elements of $X$.
Our strongest stability result states that if
$2 h(Y, X) + N_k(X,Y) < s_{i+1} -s_i$ 
for all $i$, then the poset of branch points of $\pi_0 L_{-, 0}(X)=V_{-}(X)$ is a retract of the poset of layer points of $\pi_0 L_{-,k}(Y)$ (Corollary~\ref{cor:smallparam}). 
We show this by constructing an interleaving of the form
\[\begin{tikzcd}
\Lambda(L_0 X) \ar[r, "id"] \ar[d] & \Lambda(L_0 X) \ar[d] \\
\Lambda(L_k Y) \ar[r] \ar[ur] & \Lambda(L_k Y) 
\end{tikzcd}\]
where the upper triangle commutes and the lower triangle commutes up to homotopy. 
The existence of this diagram follows from a more general result, which requires the truncated version of the hierarchical clustering, $L_0 X$.  

The hierarchical clustering $L_0 X[c]$ is defined so $L_0 X[c](s)$ is empty if $s<c$, but is otherwise equal to $L_0 X(s)$.  
In Theorem~\ref{theorem:main}, we prove the existence of a homotopy commutative diagram of posets, 
\[\begin{tikzcd}[column sep= large]
\Lambda(L_0 X[c]) \ar[d, "i_{\epsilon}"] \ar[r, "(id_Y)_{\epsilon+\delta}"]& \Lambda(L_0 X[c]) \ar[d, "i_{\epsilon}"]\\
\Lambda(L_k Y) \ar[r, "(id_Y)_{\epsilon+\delta}"] \ar[ur, "\theta_\delta"] & \Lambda(L_k Y),
\end{tikzcd}\]
where the maps depend on a choice of nonnegative numbers $\epsilon$, $\delta$, and $c$ satisfying $\delta \geq 2 h(Y, X)$, and $N_k(X,Y)-\epsilon \leq c \leq \delta$.

If these numbers satisfy the further condition
\[ c+\epsilon+\delta < s_{i+1} - s_i \]
for all $i$, then the map $\theta_{\delta}i_{\epsilon}: \Lambda(L_0 X[c]) \to \Lambda(L_0 X[c])$ is the identity.  This gives a retract,
\[ \theta_\delta: \Lambda(L_k Y) \to \Lambda(L_0 X[c]). \]

Consequently, if $N_k(X,Y)+2h(Y,X) < s_{i+1}-s_i$ for all $i$, there is a retract
\[ \Lambda(L_k Y) \to \Lambda(L_0 X). \]
An alternate proof of this corollary (Corollary~\ref{cor:smallparam}) can be given by showing \[\Lambda(L_0 X[c]) \cong \Lambda(L_0 X)\] for all $c$ satisfying $c< s_{i+1}-s_i$ for all $i$.
The maps in the interleaving of $\Lambda(L_0 X)$ and $\Lambda(L_k Y)$ of Corollary~\ref{cor:smallparam} could thus be given piecewise without defining the truncated hierarchical clustering, $L_0 X[c]$.  However, these piecewise maps are still essentially passing to the truncated clustering, and using $L_0 X[c]$ allows for the weaker statement in Theorem~\ref{theorem:main}, where $\Lambda(L_0 X[c])$ and $\Lambda(L_0 X)$ are not necessarily isomorphic.

\section{Layer Points}

\subsection{Layers}\label{subsec:background}
Let $Z$ be any set.  A clustering of $Z$ is a collection of disjoint, non-empty subsets of $Z$, called clusters.  The clusters need not cover $Z$, and, in particular, a clustering may be empty. Let $C(Z)$ be the set of clusterings of $Z$ with the partial order where $A \leq B$ if each cluster in $A$ is contained in a (unique) cluster of $B$. 

We will consider families of clusterings indexed by partially ordered sets of a particular form. 
\begin{defn}
Take $\vec{v} \in \{-1, 1\}^{n}$ for some $n \geq 1$ and let $R^{\vec{v}}$ be the product poset \[R_{v_1} \times R_{v_2}  \times \cdots \times R_{v_n}\] where $R_1 = \RR$ and $R_{-1}= \RR^{op}$.
A $\vec{v}$-hierarchical clustering  of $Z$ is an order preserving map
\[ H : R^{\vec{v}} \to C(Z), \]
such that $H(\vec{s})=\emptyset$ for some $\vec{s} \in R^{\vec{v}}$.
\end{defn}
Equivalently, if the posets $\RR^n$ and $C(Z)$ are regarded as categories,  $H$ is a functor from $\RR^n$ to $C(Z)$ that is covariant in the $i^{th}$ parameter if $v_i=1$ and contravariant in the $i^{th}$ parameter if $v_i=-1$.
\begin{note}
This is slightly different from the original definition of a $\vec{v}$-hierarchical clustering in \cite{RS}, where the domain of each parameter is $\RR_{>0}$, not $\RR$.  Since we require $H(\vec{s})=\emptyset$ for some $\vec{s}$ (and hence $H(\vec{t})=\emptyset$ for all $\vec{t} \leq {\vec{s}}$), these definitions are equivalent up to a shift in parameters.
\end{note}

For any $\vec{v}$-hierarchical clustering, $H$, we define a poset, $\Gamma(H)$, as in \cite{Branches}.  
The elements of $\Gamma(H)$ are all pairs $(\vec{s}, S)$ with $\vec{s} \in R^{\vec{v}}$ and $S$ a cluster in $H(\vec{s})$.  
There is a comparison $(\vec{s}, S) \leq (\vec{t}, T)$ if $\vec{s} \leq \vec{t}$ and $S \subseteq T$.  
This is the restriction of the product partial order on $R^{\vec{v}} \times \mathcal{P}(Z)$.

\begin{defn}
Let $\sim$ be the equivalence relation on $\Gamma(H)$ generated by the relation $(\vec{s}, S) \sim (\vec{t}, T)$ if $\vec{s} \leq \vec{t}$ and $S=T$.
An equivalence class of $\sim$ is called a {layer} of $\Gamma(H)$.
\end{defn}
Each layer
is of the form 
\[ \gamma= \{ (\vec{s}, S) | \vec{s} \in P\} \]
for some $P \subseteq R^{\vec{v}}$. 
We call the set $P \subseteq R^{\vec{v}}$ the support of $\gamma$ and may write $\gamma=(P, S)$.  

Supports of layers must satisfy the following property.   
\begin{prop}\label{prop:updown}
If $P \subseteq R^{\vec{v}}$ is the support of a layer for some $\vec{v}$-hierarchical clustering, then $P$ is equal to the intersection of the up-set of $P$,
\[ U_P := \bigcup_{\vec{p} \in P} \{ \vec{r} \in R^{\vec{v}} | \vec{r} \geq \vec{p} \} \]
and the down-set of $P$,
\[ D_P := \bigcup_{\vec{p} \in P} \{ \vec{r} \in R^{\vec{v}} | \vec{r} \leq \vec{p} \}. \]
\end{prop}
\begin{proof}
Suppose $(P, S)$ is a layer in $\Gamma(H)$ for some $\vec{v}$-hierarchical clustering, $H$.  
By definition, $U_P$ and $D_P$ both contain all elements of $P$.   If $\vec{r} \in U_P \cap D_P$, then there exist $\vec{p}, \vec{q} \in P$ such that $\vec{p} \leq \vec{r} \leq \vec{q}$. 
Consequently, $H(\vec{p}) \leq H(\vec{r}) \leq H(\vec{q})$. Since $S \in H(\vec{p})$, this means there exist unique clusters $S' \in H(\vec{r})$ and $S'' \in H(\vec{q})$ satisfying
\[ S \subseteq S' \subseteq S'' \]
However, $S \in H(\vec{q})$, so $S''$ must be equal to $S$, and therefore $S'$ is equal $S$ as well.  This proves $(\vec{r}, S) \in \Gamma(H)$ exists and $(\vec{p}, S) \sim (\vec{r}, S)$.  So, $(\vec{r},S) \in (P,S)$ and hence $\vec{r} \in P$.  
\end{proof}
When the clustering has a single parameter, the supports of layers are all intervals.  
A non-example and an example of sets that could potentially support a layer of a two-parameter clustering are depicted in Figure~\ref{fig:layer_region}.

\begin{figure}
    \centering
    \begin{tikzpicture}[scale=0.8]
    \fill[gray!20] (-6,1) circle (1cm);
        \fill[gray!20] (1,1) arc
    [
        start angle=0,
        end angle=90,
        x radius=1cm,
        y radius =1cm
    ] ;
    \fill[gray!20] (0,2) arc
    [
        start angle=180,
        end angle=270,
        x radius=1cm,
        y radius =1cm
    ] ;
    \fill[gray!20] (2,-1)--(2,0)--(1,0)--(3,0)--(3,-.5)--(4,-0.5)--(4,-1);
    \draw (2,0)--(2,-1)--(4,-1);
    \draw[dashed] (1,1) arc
    [
        start angle=0,
        end angle=90,
        x radius=1cm,
        y radius =1cm
    ] ;
    \draw (0,2) arc
    [
        start angle=180,
        end angle=270,
        x radius=1cm,
        y radius =1cm
    ] ;
    \draw[dashed](2,0)--(3,0)--(3,-.5)--(4,-0.5);
    \draw[<->](-1.2,-1.5)--(4,-1.5);
    \node[below] at (4,-1.5){$R_1$};
    \draw[<->](-1,-1.7)--(-1,3);
    \node[left] at (-1,3){$R_2$};
     \node at (0.5,1.5){$P_0$};
     \node at (2.5,-.5){$P_1$};
      \draw[<->](-8.2,-1.5)--(-4,-1.5);
    \node[below] at (-4,-1.5){$R_1$};
    \draw[<->](-8,-1.7)--(-8,3);
    \node[left] at (-8,3){$R_2$};
    \draw[gray] (-5,1) arc
    [
        start angle=0,
        end angle=90,
        x radius=1cm,
        y radius =1cm
    ] ;
    \draw[gray] (-5,1)--(-5,-1.7);
    \draw[gray] (-6,2)--(-8.2,2);
    \node[left, gray] at (-8.2,2) {$\partial D_Q$}; 
    \draw (-7,1) arc
    [
        start angle=180,
        end angle=270,
        x radius=1cm,
        y radius =1cm
    ] ;
    \draw (-7, 1)--(-7,3);
    \draw (-6,0)--(-4,0);
    \node[above] at (-7,3) {$\partial U_Q$};
    \node at (-6, 1) {$Q$};
    \fill[white] (2,-1) circle (.09cm);
    \draw (2,-1) circle (.09cm);
    \end{tikzpicture}
    \caption{The disk $Q$ cannot support a layer, as $U_Q \cap D_Q$ contains points outside of $Q$. 
    The shaded sets in the right diagram, $P_0$ and $P_1$, each satisfy the requirements to support a layer. }
    \label{fig:layer_region}
\end{figure}
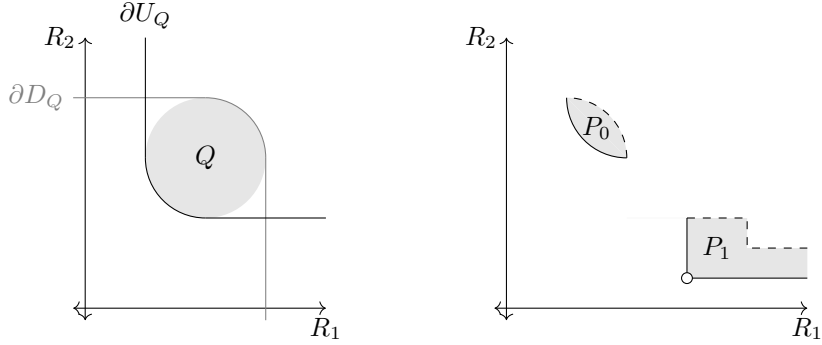

Note that there may be layers of the form $(P, S)$ and $(Q, S)$ for some distinct $P$ and $Q$.  In this case, any $p \in P$ and $q \in Q$ are incomparable.   (Otherwise, $(\vec{p}, S)$ and $(\vec{q}, S)$ would be in the same equivalence class.) This implies the following corollary of Proposition~\ref{prop:updown}.

\begin{cor}
Let $H$ be a $\vec{v}$-hierarchical clustering of $Z$.  For any $S \subseteq Z$, let
\[ Supp(S) = \{ \vec{s} \in R^{\vec{v}} | S \in H(\vec{s}) \}. \]
Then $Supp(S) = U_{Supp(S)} \cap D_{Supp(S)}$.
\end{cor}
\begin{proof}
The set $Supp(S)$ is the union of a family of disjoint sets $\{P_i\}_{i \in I}$ where each $P_i$ is the support of a distinct layer, $(P_i, S)$. For any $\vec{r} \in U_{P_i} \cap D_{P_j}$, there exist $\vec{p} \in P_i$ and $\vec{q} \in P_j$ such that $\vec{p} \leq \vec{r} \leq \vec{q}$.  If $i \neq j$, no $\vec{p} \in P_i$ and $\vec{q} \in P_j$ are comparable, so $U_{P_i} \cap D_{P_j} $ is empty.  Hence,
\[ Supp(S) = \bigcup_{i \in I} P_i = \bigcup_{i \in I} (U_{P_i} \cap D_{P_i}) =  \left(\bigcup_{i \in I} U_{P_i} \right) \cap  \left(\bigcup_{j \in I} D_{P_j} \right) = U_{Supp(S)} \cap D_{Supp(S)}. \]
\end{proof}
In Figure~\ref{fig:layer_region}, $P_0$ and $P_1$ could potentially support layers associated to the same cluster.
Not every $P \subseteq R^{\vec{v}}$ satisfying $P = U_P \cap D_P$ will necessarily be equal to $Supp(S)$ for some cluster, $S$, of an arbitrary $\vec{v}$-hierarchical clustering.  However, for any such $P$, if $Z$ has a proper subset $W$, then there exists at least one $\vec{v}$-hierarchical clustering of $Z$ where $P=Supp(W)$.
One such hierarchical clustering is given by
\[ H(\vec{s}) = 
\begin{cases}
\emptyset &= \vec{s} \in R^{\vec{v}} \setminus U_P \\
\big\{W\big\} &=  \vec{s} \in P \\
\big\{Z\big\} &= \vec{s} \in U_P \setminus P. \\
\end{cases}\]

From here on, we restrict to $\vec{v}$-hierarchical clusterings satisfying the following condition.  

Consider $R^{\vec{v}}$ as the topological space $\RR^n$, and let $\partial X$ denote the boundary of $X \subseteq R^{\vec{v}}$.
We define the lower boundary of $X$ to be
\[ D_X \cap \partial X. \]
\begin{defn}\label{defn:bbdbelow}
A set $X \subseteq R^{\vec{v}}$ is closed below if the lower boundary of $X$ is contained in $X$.

A $\vec{v}$-hierarchical clustering, $H$, is closed below if the support of every layer of $\Gamma(H)$ is closed below.
\end{defn}

For example, in Figure~\ref{fig:layer_region}, $P_0$ is closed below, but $P_1$ is not.  

Restricting to clusterings that are closed below will be used to guarantee the existence of minimal points within a layer.  These minima give the locations of layer points. 

\begin{defn}
Let $H$ be a $\vec{v}$-hierarchical clustering of a set $Z$.  An element $(\vec{s}, S)$ of the poset $\Gamma(H)$ is a (global) layer point if
it satisfies one of the following conditions:
\begin{enumerate}[label=\roman*.)]
    \item there is no $(\vec{t}, T) \in \Gamma(H)$ with $(\vec{t}, T) < (\vec{s}, S)$, or
    \item for all $(\vec{t}, T) \in \Gamma(H)$ with $(\vec{t}, T) < (\vec{s}, S)$, $|T| < |S|$.
\end{enumerate}
The set of global layer points in $\Gamma(H)$ is denoted $\Lambda(H)$. 
\end{defn}
If $(\vec{t}, T) \leq (\vec{s}, S)$, then $T \subseteq S$. In the second case, the condition $|T|<|S|$ is therefore equivalent to $T \neq S$.  
So, the layer points are precisely the points in $\Gamma(H)$ where a cluster changes.  The poset $\Gamma(H)$, or equivalently the graph it defines, thus provides a condensed version of the information needed to define the $\vec{v}$-hierarchical clustering, $H$.  

A layer point can also be described as an element $(\vec{s}, S) \in (P,S)\subseteq \Gamma(H)$ where $D_{\{\vec{s}\}} \cap P = \{\vec{s}\}$.  For every layer point $(\vec{s}, S)$, $\vec{s}$ lies on the lower boundary of the support of the associated layer.  

\begin{exmp}
Suppose there is a layer, $(P, S)$ in $\Gamma(H)$ where $P$ is the shaded region in 
Figure~\ref{fig:layer_points}.  An element $(\vec{s}, S) \in (P,S)$ is a layer point if $\vec{s}$ is $\vec{s}_2$ or any point on the portion of the lower boundary of $P$ between $\vec{s}_0$ and $\vec{s_1}$ (shown in bold).    
\end{exmp}
\begin{figure}
    \centering
    \begin{tikzpicture}
     \fill[gray!20] (0.7,2.5) arc
    [
        start angle=180,
        end angle=270,
        x radius=1cm,
        y radius =1cm
    ] ;
        \fill[gray!20] (0.7,3)--(0.7, 2.5)--(1.7, 1.5)--(2.5, 1.5)--(2.5,0.5)--(3,0.5)--(3,3);
    \draw[<->] (-0.2,0)--(3,0);
    \node[below] at (3,0) {$R_1$};
    \draw[<->] (0, -0.2)--(0,3);
    \node[left] at (0,3) {$R_2$};
    \node at (2,2.2) {$P$};
    \draw[very thick] (0.7,2.5) arc
    [
        start angle=180,
        end angle=270,
        x radius=1cm,
        y radius =1cm
    ] ;
    \draw (0.7, 2.5)--(0.7, 3);
    \draw (1.7,1.5)--(2.5,1.5);
    \draw (2.5,1.5)--(2.5,0.5)--(3,0.5);
    \fill (2.5,0.5) circle (2pt);
    \fill (1.7,1.5) circle (2pt);
    \fill (0.7, 2.5) circle (2pt);
    \node[left] at (0.7, 2.5) {$\vec{s}_0$};
    \node[below] at (1.7,1.5) {$\vec{s}_1$};
    \node[left] at (2.5,0.5) {$\vec{s}_2$};
    \end{tikzpicture}
    \caption{The shaded area, $P$, supports a layer, $(P, S)$. The bold points indicate the locations of global layer points associated to $S$.}
    \label{fig:layer_points}
\end{figure}
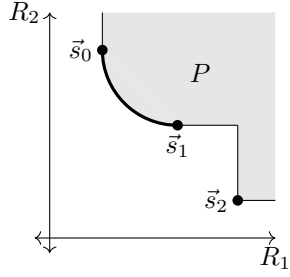

An antichain is a subset of a poset that contains no pairs of comparable elements. 
By definition, any two global layer points associated to the same cluster are incomparable.  So, the set of global layer points in a single layer form an antichain in $\Gamma(H)$. 
This defines a map from $\Gamma(H)$ to the collection of antichains in $\Gamma(H)$ that takes $(\vec{s}, S)$ to the set of global layer points in the layer containing $(\vec{s}, S)$.  
If any one of these global layer points is comparable to $(\vec{s}, S)$, then it is a maximal layer point below $(\vec{s}, S)$.  
That is, if $(\vec{r}, S)$ is a global layer point associated to $S$ and $(\vec{r}, S)$ is comparable to some $(\vec{s}, S)$, then $(\vec{r}, S) \leq (\vec{s}, S)$ and for any other layer point $(\vec{t}, T) \in \Lambda(H)$ comparable to $(\vec{s}, S)$, $(\vec{t}, T) \not > (\vec{r}, S)$.  
In the following section, we will show that there is at least one global layer point comparable to $(\vec{s}, S)$.  

For a single parameter clustering, the lower boundary of each layer is a point, so there is exactly one layer point corresponding to each layer.  Furthermore, this layer point is comparable to all points in the layer. Thus, there is a unique maximal layer point below any element of $\Gamma(H)$.   This defines a map
\[ m:\Gamma(H) \to \Lambda(H). \]
In Section~\ref{sec:layer_stab}, we will use this map to induce interleavings of layer points from interleavings of clusterings. 

\subsection{Slice Clusterings}

Single parameter clusterings can be extracted from a given multi-parameter clustering by taking one-dimensional slices. 
Recall that for any $\vec{v} \in \{-1,1\}^n$, $R^{\vec{v}} = \prod\limits_{i=1}^n R_{v_i}$ where \[R_{v_i} = \begin{cases} \RR & v_i = 1 \\ \RR^{op} & v_i=-1. \end{cases}\] 
For any $1 \leq i \leq n$, let $C_i = R_{v_1} \times \cdots \times R_{v_{i-1}} \times R_{v_{i+1}} \times \cdots \times R_n$.  
Let $p_i: R^{\vec{v}} \to C_i$ be the projection map, \begin{align*}
(r_1, \ldots, r_n) & \mapsto (r_1, \ldots, r_{i-1}, r_{i+1}, \ldots, r_n)
\end{align*} and let $\phi_i: R_{v_i} \times C_i \to R^{\vec{v}}$ be the insertion map, \begin{align*}
\big(t, (c_1, \ldots, c_{n-1})\big) & \mapsto (c_1, \ldots, c_{i-1}, t, c_{i}, \ldots, c_{n-1}).
\end{align*} 
By construction, these maps satisfy the properties $\phi_i(r_i, p_i \vec{r}) = \vec{r}$ for all $\vec{r} \in R^{\vec{v}}$ and $p_i(\phi_i(t, \vec{c}))=\vec{c}$ for all $(t, \vec{c}) \in R_{v_i} \times C_i$.  
\begin{defn}
For any $1 \leq i \leq n$ and any $\vec{w} \in C_i$ we define the $i^{th}$-parameter slice of $H$ at $\vec{w} \in C_i$ to be the hierarchical clustering
\[ H_{\vec{w}}^i : R_{v_i} \to C(Z) \]
that takes $s$ to $H(\phi_i(s, \vec{w}))$.\end{defn} 
Take $\Gamma_{\vec{w}}^i(H)$ to be the subset of $\Gamma(H)$ consisting of all $(\vec{s}, S) \in \Gamma(H)$ with $p_i(\vec{s})=\vec{w}$. 
There is an isomorphism of posets, \[ \Phi: \Gamma_{\vec{w}}^i(H)  \to \Gamma(H_{\vec{w}}^i)\]
induced by the projection of $R^{\vec{v}}$ onto $R_{v_i}$.

We can consider the layer points of the slice clusterings within the larger context of the multiparameter clustering. 
\begin{defn}
Let $H$ be a $\vec{v}$-hierarchical clustering of a set $Z$.  An element $(\vec{t}, T)$ of the poset $\Gamma(H)$ is a $i^{th}$-parameter layer point if it is the image of the inclusion
\[ \Lambda(H_{\vec{s}}^i) \to \Gamma(H_{\vec{s}}^i) \xrightarrow{\cong} \Gamma_{\vec{s}}^i(H) \to \Gamma(H)  \]
for some $\vec{s}$.  We denote the set of all $i^{th}$-parameter layer points by $\Lambda^i(H)$. 
\end{defn}

For any layer, $(P_i, S)$, in $\Gamma_{\vec{w}}^i(H)$, let $(P, S)$ be the layer associated to $S$ in $\Gamma(H)$.  Then $P_i = P \cap p_i^{-1}(\vec{w})$.  This means if  $(\vec{s}, S)$ is the (unique) $i^{th}$-parameter layer point associated to $S$ in $\Gamma_{\vec{w}}^i(H)$, then $\vec{s}$ lies on the lower boundary of $P$.   

Figure~\ref{fig:layer_points_slices} gives an example of $i^{th}$-parameter layer points.  Comparing with Figure~\ref{fig:layer_points}, we see that the global layer points associated to the layer supported by this region are the points that are both $1^{st}$- and $2^{nd}$-parameter layer points.  In fact, for any number of parameters, the global layer points can be found by taking the intersection of the $i^{th}$-parameter layer points for all $i$.

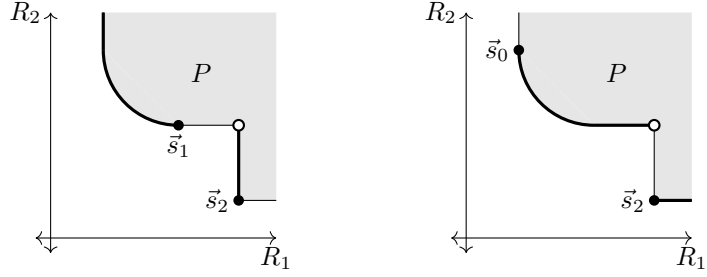
\begin{figure}
    \centering
    \phantom{.}\hfill\begin{tikzpicture}
     \fill[gray!20] (0.7,2.5) arc
    [
        start angle=180,
        end angle=270,
        x radius=1cm,
        y radius =1cm
    ] ;
        \fill[gray!20] (0.7,3)--(0.7, 2.5)--(1.7, 1.5)--(2.5, 1.5)--(2.5,0.5)--(3,0.5)--(3,3);
    \draw[<->] (-0.2,0)--(3,0);
    \node[below] at (3,0) {$R_1$};
    \draw[<->] (0, -0.2)--(0,3);
    \node[left] at (0,3) {$R_2$};
    \draw[very thick] (0.7,2.5) arc
    [
        start angle=180,
        end angle=270,
        x radius=1cm,
        y radius =1cm
    ] ;
    \draw[very thick] (0.7, 2.5)--(0.7, 3);
    \draw[very thick] (2.5,1.5)--(2.5,0.5);
    \draw (0.7, 2.5)--(0.7, 3);
    \draw (1.7,1.5)--(2.5,1.5);
    \draw (2.5,1.5)--(2.5,0.5)--(3,0.5);
    \fill (2.5,0.5) circle (2pt);
    \fill[white] (2.5, 1.5) circle (2pt);
    \draw[thick] (2.5, 1.5) circle (2pt);
    \fill (1.7,1.5) circle (2pt);
    \node[below] at (1.7,1.5) {$\vec{s}_1$};
    \node[left] at (2.5,0.5) {$\vec{s}_2$};
    \node at (2,2.2) {$P$};
    \end{tikzpicture} \hfill 
        \begin{tikzpicture}
     \fill[gray!20] (0.7,2.5) arc
    [
        start angle=180,
        end angle=270,
        x radius=1cm,
        y radius =1cm
    ] ;
        \fill[gray!20] (0.7,3)--(0.7, 2.5)--(1.7, 1.5)--(2.5, 1.5)--(2.5,0.5)--(3,0.5)--(3,3);
    \draw[<->] (-0.2,0)--(3,0);
    \node[below] at (3,0) {$R_1$};
    \draw[<->] (0, -0.2)--(0,3);
    \node[left] at (0,3) {$R_2$};
    \draw[very thick] (0.7,2.5) arc
    [
        start angle=180,
        end angle=270,
        x radius=1cm,
        y radius =1cm
    ] ;
    \draw[very thick] (1.7, 1.5)--(2.5,1.5);
    \draw (0.7, 2.5)--(0.7, 3);
    \draw (1.7,1.5)--(2.5,1.5);
    \draw (2.5,1.5)--(2.5,0.5)--(3,0.5);
    \fill (2.5,0.5) circle (2pt);
    \fill[white] (2.5, 1.5) circle (2pt);
    \draw[thick] (2.5,1.5) circle (2pt);
    \fill (0.7, 2.5) circle (2pt);
    \node[left] at (0.7, 2.5) {$\vec{s}_0$};
    \node[left] at (2.5,0.5) {$\vec{s}_2$};
    \draw[very thick] (2.5, 0.5)--(3,0.5);
    \node at (2,2.2) {$P$};
    \end{tikzpicture} \hfill \phantom{.}
    \caption{The shaded area supports a layer, $(P, S)$. The bold points in the left diagram indicate the locations of 1st-parameter layer points associated to $S$.  The bold points in the right diagram indicate the locations of 2nd-parameter layer points associated to $S$.}
    \label{fig:layer_points_slices}
\end{figure}

\begin{prop}
An element $(\vec{s}, S) \in \Gamma(H)$ is a global layer point if and only if it is an $i^{th}$-parameter layer point for all $i$.
\end{prop}
\begin{proof}
Suppose $(\vec{s}, S)$ is a global layer point in $\Gamma(H)$.  We will show $(s_i, S)$ is a layer point for the $i^{th}$-parameter slice of $H$ at $p_i \vec{s}$.
For any $(t, T) \leq (s_i, S)$ in $\Gamma(H_{p_i \vec{s}}^i)$, \[ \phi_i (t, p_i\vec{s}) \leq \phi_i (s_i, p_i \vec{s}) = \vec{s} \]
and $T \in H_{p_i\vec{s}}^i (t) = H(\phi_i(t, p_i\vec{s}))$ with $T \subseteq S$, so 
\[  (\phi_i (t, p_i\vec{s}), T) \leq (\vec{s}, S).  \]
If $t \neq s_i$, then $\phi_i (t, p_i\vec{s}) \neq \vec{s}$, so $|T| < |S|$.  
Since this is true for any  $(t, T) < (s_i, S)$, the point $(s_i, S) \in \Gamma(H_{\vec{s}}^i)$ is a layer point.  
The image of $(s_i, S)$ under the inclusion of $\Lambda(H_{p_i\vec{s}}^i)$ into $\Gamma(H)$ is $\left(\phi_i (s_i, p_i \vec{s}), S \right)=(\vec{s}, S)$, so $(\vec{s}, S)$ is an $i^{th}$-parameter layer point. 

Now suppose $(\vec{s}, S)$ is an $i^{th}$-parameter layer point for all $i$.  For any $(\vec{t}, T) \leq (\vec{s}, S)$ with $\vec{t} \neq \vec{s}$, there is at least one $i$ such that $t_i < s_i$.  
Therefore,  \[ \vec{t} \leq \phi_i(t_i, p_i\vec{s}) < \vec{s}. \]
Let $\overline{T}$ be the cluster in $H(\phi_i(t_i, p_i\vec{s}))$ containing $T$.  Then, in $\Gamma(H_{p_i \vec{s}}^i)$, 
\[(t_i, \overline{T} ) < (s_i, S).\]
Since $(\vec{s}, S)$ is an $i^{th}$-parameter layer point, this implies $\lvert\overline{T}\rvert < |S|$, and consequently, $|T|<|S|$.  Hence, $(\vec{s}, S)$ is a global layer point. 
\end{proof}

Recall that the maximal global layer point below some $(\vec{s}, S) \in \Gamma(H)$ is not necessarily unique in the multiparameter case.  
Instead, the global layer points in the layer associated to $S$ form an antichain in $\Gamma(H)$.  
However, since the slice clusterings each have a single parameter, there is a unique maximal $i^{th}$-parameter layer point below $(\vec{s}, S)$.  This defines a map
\begin{align*}
    m_i:\Gamma(H) &\to \Lambda^i(H) 
\end{align*}
for each $i$ that takes $(\vec{s}, S)$ to the maximal $i^{th}$-parameter layer point below $(\vec{s}, S)$. 

These maps to $i^{th}$-parameter layer points are compatible in that the maximal $j^{th}$-parameter layer point below an $i^{th}$-parameter layer point is also an $i^{th}$-parameter layer point. 
\begin{lem}\label{lem:compose}
If $(\vec{s}, S) \in \Lambda^i(H)$ then $m_j (\vec{s}, S) \in \Lambda^i(H)$ for all $1 \leq j \leq n$.
\end{lem}
\begin{proof}
Let $(\vec{t}, S) := m_j(\vec{s}, S)$. 
For any $\vec{w} \leq \vec{t}$ in $R^{\vec{v}}$ with $p_i \vec{w} = p_i \vec{t}$, we define $\vec{q}=\phi_i(w_i, p_i \vec{s})$.  (The positions of these four points in $R^{\vec{v}}$ are depicted in Figure~\ref{fig:square}.) Then we have
\begin{align*}
    p_j \vec{q} &= p_j\vec{w} \\
    q_j = s_j &\geq t_j = w_j \\
    p_i\vec{q} &= p_i \vec{s} \\
    q_i = w_i &\leq t_i = s_i.
\end{align*}
For any $W \in H(\vec{w})$ with $W \subseteq S$, let $Q$ be the cluster in $H(\vec{q})$ containing $W$.  

Suppose $W = S$.  Then $S \subseteq Q$.  Since $(\vec{s}, S)$ is an $i^{th}$-parameter layer point, $p_i\vec{s}=p_i\vec{q}$, and $q_i \leq s_i$, this implies $(\vec{q}, Q) = (\vec{s}, S)$.  So, $w_i = q_i = s_i = t_i$.  Since $p_i \vec{w} = p_i \vec{t}$, $(\vec{w}, W)=(\vec{t}, S)$.  

Hence, if $(\vec{w}, W) \in \Gamma_{p_i \vec{t}} \ (H)$ and $(\vec{w}, W) < (\vec{t}, S)$, then $W \neq S$.
That is, $m_j(\vec{s}, S)=(\vec{t}, S)$ is an $i^{th}$-parameter layer point. 
\end{proof}
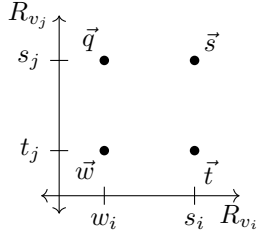
\begin{figure}[h]
    \centering
    \begin{tikzpicture}[scale=0.6]
         \draw[<->] (-0.4, 0)--(4,0);
         \draw[<->] (0, -0.4)--(0,4);
         \node[below] at (4,0) {$R_{v_i}$};
         \node[left] at (0,4) {$R_{v_j}$};
         \fill (1,1) circle (3pt);
         \node[below left] at (1,1) {$\vec{w}$};
         \fill (3,1) circle (3pt);
         \node[below right] at (3,1) {$\vec{t}$};
         \fill (1,3) circle (3pt);
         \node[above left] at (1,3) {$\vec{q}$};
         \fill (3,3) circle (3pt);
         \node[above right] at (3,3) {$\vec{s}$};
         \draw (-0.2, 1)--(.2,1);
         \node[left] at (-.2, 1) {$t_j$};
         \draw (-0.2, 3)--(.2,3);
         \node[left] at (-.2, 3) {$s_j$};
         \draw (1,-.2)--(1,.2);
         \node[below] at (1,-.2) {$w_i$};
         \draw (3,-.2)--(3,.2);
         \node[below] at (3,-.2) {$s_i$};
    \end{tikzpicture}
    \caption{The points $\vec{s}$, $\vec{t}$, $\vec{w}$, and $\vec{q}$ form a square in the 2-parameter slice of $R^{\vec{v}}$, $R_{v_i} \times R_{{v_j}}$.}
    \label{fig:square}
\end{figure}

Composing the maps $m_i$ for all $i$ therefore gives a retract
\[ m: \Gamma(H) \to \Lambda(H). \]
However, the composition of a given $m_i$ and $m_j$ is not necessarily commutative, so this map depends on a choice of order.  Additionally, even when the order of the maps is fixed, the global layer point selected is not necessarily the same for two elements in the same layer.

\subsection{Branch Points}
For a single parameter clustering, the layer set contains the branch points, as defined in \cite{Branches}.  Heuristically, branch points indicate where clusters originate or combine, but not where clusters grow in size without merging with another cluster. 
The branch points can been seen as indicating the births and deaths of persistent clusters (where persistent clusters are defined as in \cite{RS}).  

\begin{defn}
Let $H$ be a single parameter hierarchical clustering of a set $Z$.  An element $(s, S)$ of the poset $\Gamma(H)$ is a branch point if it satisfies one of the following conditions:
\begin{enumerate}[label=\roman*.)]
    \item there is no $(t, T) \in \Gamma(H)$ with $(t, T) < (s, S)$, or
    \item there exist distinct $(t, T_0)$ and $(t, T_1)$ in $\Gamma(H)$ such that $(t, T_0) \leq (r, R)$ and $(t, T_1) \leq (r, R)$ if and only if $(r, R) \geq (s, S)$.
\end{enumerate}
\end{defn}

For a $\vec{v}$-hierarchical clustering, $H$, $i^{th}$-parameter branch points are defined as the branch points of the slice hierarchical clusterings, in the same way as $i^{th}$-parameter layer points were defined as the layer points of the slice clusterings.  These appear in \cite{Layers} as vertical and horizontal branch points.

Unlike with layer points, an element of $\Gamma(H)$ does not need to be an $i^{th}$ parameter branch point for all $i$ to satisfy a generalized version of Definition 6.  
This results in some ambiguity with regard to what a global branch point should be, but we propose the following definition.  
\begin{defn}
An element of $\Gamma(H)$ is a global branch point if it is an $i^{th}$-parameter branch point for all $i$.
\end{defn}

Unfortunately, the analogue of Lemma~\ref{lem:compose} does not hold for $i^{th}$-parameter branch points, so composing maps to slice branch points will not necessarily yield a global branch point. We do not focus on the theory of branch points here, but we will utilize them in Section~\ref{sec:layer_stab} where the layer points of a particular hierarchical clustering are all branch points.

\subsection{Interleavings}
\label{sec:interleavings}
Suppose $Z$ is a set, $f$ is a self-map on $Z$, and $H$ and $E$ are $\vec{v}$-hierarchical clusterings of $Z$.  
As in \cite{RS}, we put $\vec{s}+\vec{v}\vec{\epsilon}$ to mean $(s_1+v_1 \epsilon_1, \ldots, s_n + v_n\epsilon_n)$ for any $\vec{\epsilon} \in \RR_{\geq 0}^n$. 
If, there exists, for all $S \in H(\vec{s})$, some $T \in E(\vec{s}+\vec{v}\vec{\epsilon})$ such that $f(S) \subseteq T$, then there exists an induced map \[f_{\scaleto{\vec{\epsilon}}{7pt}}:H(\vec{s}) \to E(\vec{s}+\vec{v}\vec{\epsilon})\] that takes a cluster $S$ to the cluster containing $f(S)$, which we denote by $\overline{f(S)}$.  When $f_{\scaleto{\vec{\epsilon}}{7pt}}$ exists, it is natural in $\vec{s}$.

We say there is an $(\vec{\epsilon}, \vec{\delta})$-interleaving of the pair $(H, E)$ given by $(f,g)$ if, for every $\vec{s} \in R^{\vec{v}}$, all maps in the following diagrams exist and the diagrams commute for all $\vec{s}$:
\[\adjustbox{scale=0.99}{\begin{tikzcd}[column sep=tiny, scale=0.5]
H(\vec{s}) \ar[rr, "{(id_Y)_{\scaleto{\vec{\epsilon}+\vec{\delta}}{9pt}}}"] \ar[dr, "{f_{\scaleto{\vec{\epsilon}}{6pt}}}", swap] 
& & H(\vec{s}+\vec{v}(\vec{\epsilon}+\vec{\delta}))   \\
& E(\vec{s}+\vec{v}\vec{\epsilon}) \ar[ur,"{g_{\scaleto{\vec{\delta}}{7pt}}}", swap]  & 
\end{tikzcd}  \begin{tikzcd}[column sep=tiny, scale=0.5]
 & H(\vec{s}+\vec{v}\vec{\delta}) \ar[dr, "{f_{\scaleto{\vec{\epsilon}}{6pt}}}"] & \\
 E(\vec{s}) \ar[ur,"{g_{\scaleto{\vec{\delta}}{7pt}}}"] \ar[rr, "{(id_Y)_{\scaleto{\vec{\epsilon}+\vec{\delta}}{9pt}}}", swap] & & E(\vec{s}+\vec{v}(\vec{\epsilon}+{\delta})).
\end{tikzcd}}\]  
This is a special case of the interleavings of generalized persistence modules in \cite{BDS}.
  In the case where $\vec{\epsilon}=\vec{\delta}$ and $f=g=id_Y$, this coincides with the notion of an $\vec{\epsilon}$-interleaving in \cite{RS}.   

An $(\vec{\epsilon}, \vec{\delta})$-interleaving of $(H, E)$ induces a commutative diagram of posets,
\[\begin{tikzcd}[column sep=large]
\Gamma(H) \ar[d, "{f_{\scaleto{\vec{\epsilon}}{6pt}}}", swap] \ar[r, "{(id_Z)_{\scaleto{\vec{\epsilon}+\vec{\delta}}{9pt}}}"] & \Gamma(H)\ar[d, "{f_{\scaleto{\vec{\epsilon}}{6pt}}}"] \\
\Gamma(E) \ar[r, "{( id_Z)_{\scaleto{\vec{\epsilon}+\vec{\delta}}{9pt}}}", swap] \ar[ur, "{g_{\scaleto{\vec{\delta}}{7pt}}}"] & \Gamma(E)
\end{tikzcd}\]
where 
\begin{align*}
    f_{\scaleto{\vec{\epsilon}}{6pt}}(\vec{s}, S) &= \left( \vec{s}+\vec{v}\vec{\epsilon}, \ \overline{f(S)}\right) \\
    g_{\scaleto{\vec{\delta}}{7pt}}(\vec{s}, S) &= \left( \vec{s}+\vec{v}\vec{\delta}, \ \overline{g(S)}\right) \\
    (id_Z)_{\scaleto{\vec{\epsilon}+\vec{\delta}}{9pt}}(\vec{s}, S) &= \left(\vec{s}+\vec{v}\left(\vec{\epsilon} + \vec{\delta}\right), \ \overline{S}\right).
\end{align*}
In general, interleavings of $(H,E)$ do not induce interleavings of the slice clusterings, $(H_{\vec{s}}^i, E_{\vec{s}}^i)$, since the maps in the interleaving may move between slices. 
However, if $\vec{\epsilon}=\phi_i(\epsilon, \vec{0})$ and $\vec{\delta}=\phi_i(\delta, \vec{0})$ for some $i$, then the maps restrict to $i^{th}$-directional slices.  In that case, there is an $(\vec{\epsilon}, \vec{\delta})$-interleaving of $(H, E)$ given by $(f,g)$ if and only if there is an $(\epsilon, \delta)$-interleaving of $(H_{\vec{s}}^i, E_{\vec{s}}^i)$ given by $(f,g)$ for all $\vec{s} \in C_i$.

An interleaving of $\vec{v}$-hierarchical clusterings induces a homotopy interleaving of layer points.  In the multi-parameter case, this depends on a choice of maps $m_H$ and $m_E$.
\begin{theorem}\label{thm:stability1}
An $(\vec{\epsilon}, \vec{\delta})$-interleaving of a pair of $\vec{v}$-hierarchical clusterings, $(H, E)$, given by $(f,g)$ induces a homotopy commutative diagram, 
\[  \begin{tikzcd}[column sep = large]
\Lambda(H) \ar[d, "f_{\scaleto{\vec{\epsilon}}{6pt}}", swap] \ar[r, "(id_Z)_{\scaleto{\vec{\epsilon}+\vec{\delta}}{9pt}}"] & \Lambda(H)\ar[d, "f_{\scaleto{\vec{\epsilon}}{6pt}}"]\\
\Lambda(E) \ar[r, "(id_Z)_{\scaleto{\vec{\epsilon}+\vec{\delta}}{9pt}}", swap] \ar[ur, "g_{\scaleto{\vec{\delta}}{6pt}}"] & \Lambda(E).\end{tikzcd}\]
\end{theorem}
\begin{proof}
Recall there exists a (non-unique) map to maximal layer points, \begin{align*}m_H:& \Gamma(H) \to \Lambda(H).\end{align*}
Let $i_H: \Lambda(H) \to \Gamma(H)$ be the inclusion.  Then 
\begin{align*}
    m_H \circ i_H = id_{\Lambda(H)} \text{ and,} \\
    i_H \circ m_H \leq id_{\Lambda(H)}. 
\end{align*}
The same properties hold for the retract $m_E$ and inclusion $i_E$.
 An $(\vec{\epsilon}, \vec{\delta})$-interleaving of $(H, E)$ given by $(f,g)$ therefore defines a diagram
 \[\begin{tikzcd}[column sep=large]
  \Lambda(H) \ar[dr] & &  \Lambda(H) & \\
& \Gamma(H) \ar[d, "{f_{\scaleto{\vec{\epsilon}}{6pt}}}", swap] \ar[r, "{(id_Z)_{\scaleto{\vec{\epsilon}+\vec{\delta}}{9pt}}}"] & \Gamma(H)\ar[d, "{f_{\scaleto{\vec{\epsilon}}{6pt}}}"] \ar[u, "m_H", swap] & \Lambda(H) \ar[l, "i_H"] \ar[ul, equal] \\
\Lambda(E) & \Gamma(E) \ar[r, "{( id_Z)_{\scaleto{\vec{\epsilon}+\vec{\delta}}{9pt}}}", swap] \ar[ur, "{g_{\scaleto{\vec{\delta}}{7pt}}}"] \ar[l, "m_E"]& \Gamma(E) \ar[rd, "m_E"] &  \\
& \Lambda(E) \ar[u, "i_E"] \ar[ul, equal] &  & \Lambda(E)
\end{tikzcd}\]
where the center squares commutes and the outer triangles commute up to a homotopy of poset maps. 
We collapse this diagram to obtain a homotopy commutative diagram,
\[  \begin{tikzcd}[column sep = large]
\Lambda(H) \ar[d, "f_{\scaleto{\vec{\epsilon}}{6pt}}", swap] \ar[r, "(id_Z)_{\scaleto{\vec{\epsilon}+\vec{\delta}}{9pt}}"] & \Lambda(H)\ar[d, "f_{\scaleto{\vec{\epsilon}}{6pt}}"] \\
\Lambda(E) \ar[r, "(id_Z)_{\scaleto{\vec{\epsilon}+\vec{\delta}}{9pt}}", swap] \ar[ur, "g_{\scaleto{\vec{\delta}}{6pt}}"] & \Lambda(E)
\end{tikzcd}\]
where, for simplicity, we omit the compositions with $i$ and $m$ in the map notation. (So, for example we put $f_{\scaleto{\vec{\epsilon}}{7pt}}$ rather than $m_E \circ f_{\scaleto{\vec{\epsilon}}{7pt}} \circ i_H$.)  \end{proof}

For branch points, there is a unique maximal branch point below any element of $\Gamma(H)$ when $H$ has a single parameter.  So, an interleaving of single parameter hierarchical clusterings also induces a homotopy interleaving of branch points. 
However, in the multi-parameter case, the situation is even worse than it is for layer points, since we have not defined a map from $\Gamma(H)$ to the branch points of $H$.  
In the remaining sections, we primarily focus on interleavings of layer points for single parameter hierarchical clusterings. 

\subsection{Approximations of Clusterings}
Let $X \subseteq Y$ be finite metric spaces.  The directional Hausdorff distance, $h$, from $Y$ to $X$ is given by 
\[  h(Y, X) = \sup_{y \in Y} \  \inf_{x \in X} d(y, x). \]

Let $H$ be a $\vec{v}$-hierarchical clustering of $X$ and let $E$ be a $\vec{v}$-hierarchical clustering of $Y$. Since $X \subseteq Y$, $H$ is also a $\vec{v}$-hierarchical clustering of $Y$. 

\begin{defn}
We say $(X, H)$ is an $(\vec{\epsilon}, \vec{\delta})$-approximation of $(Y, E)$
if there exists an $(\vec{\epsilon}, \vec{\delta})$-interleaving of the pair $(H, E)$ given by $(i, \theta)$ where $i$ is the inclusion of $X$ into $Y$ and $\theta: Y \to X$ is some map satisfying $d(y, \theta(y)) \leq h(Y,X)$.  
\end{defn}
Note that the existence of $i_{\vec{\epsilon}}$ in the $(\vec{\epsilon}, \vec{\delta})$-interleaving of $(H,E)$ is equivalent to requiring that each cluster in $H(\vec{s})$ is contained in a cluster of $E(\vec{s}+\vec{\epsilon})$.  
An $(\vec{\epsilon}, \vec{\delta})$-approximation of $(Y, E)$ by $(X, H)$ gives a homotopy commutative interleaving diagram of layer points, 
\[\begin{tikzcd}[column sep = large]
\Lambda(H) \ar[r, "(id_Y)_{\scaleto{\vec{\epsilon}+ \vec{\delta}}{8pt}}"] \ar[d, "i_{\scaleto{\vec{\epsilon}}{6pt}}", swap] & \Lambda(H) \ar[d, "i_{\scaleto{\vec{\epsilon}}{6pt}}"] \\
\Lambda(E) \ar[r, "(id_Y)_{\scaleto{\vec{\epsilon}+ \vec{\delta}}{8pt}}", swap] \ar[ur, "\theta_{\scaleto{\vec{\delta}}{7pt}}"] & \Lambda(E)
\end{tikzcd}\]
as in the previous section. This relies on the fact that $H$ and $E$ are both $\vec{v}$-clusterings of $Y$.
\section{Layer Stability for Lesnick Clusterings}
\label{sec:layer_stab}
\subsection{Approximating Lesnick Clusterings} In the remaining sections, we focus on the $\vec{v}$-hierarchical clusterings given by taking path components of Lesnick complexes, which are subcomplexes of the Vietoris-Rips complex. 

\begin{defn}
The Vietoris-Rips complex, $V_s(Z)$, is the simplicial complex with $n$-simplices of the form $[z_0, z_1, \ldots, z_n]$ for any points $z_i \in Z$ with $d(z_i, z_j)$ for all $0 \leq i, j \leq n$. \end{defn}

An $s$-neighbour of a point $z \in Z$ is a point $x \in Z$ with $x \neq z$ and $d(x,z) \leq s$.  The Lesnick complexes, or degree Rips complexes, $L_{s,k}(Z)$ filter the Vietoris-Rips complex, $V_s(Z)$, according to density, as measured by number of neighbours.  
\begin{defn} The Lesnick complex, $L_{s,k}(Z)$
has as $n$-simplices the simplices $[z_0, z_1, \ldots, z_n] \in \left(V_s(Z)\right)_n$ where each $z_i$ has at least $k$ distinct $s$-neighbours.
\end{defn}
For all $s$, note $L_{s,0}(Z)=V_{s}(Z)$.  
\begin{defn}
Let $ L:\RR \times \RR^{op} \to C(Z)$ be the map
\[ L(s,t) =  \begin{cases}  \pi_0 L_{s, \lceil t\rceil }(Z) & (s,t) \in [0, \infty) \times [0, \infty)^{op} \\ \emptyset & \text{otherwise.}   \end{cases} \]
\end{defn}

We choose to take $\lceil t \rceil$ to obtain an integer number of neighbours rather than $\lfloor t \rfloor$ in order to ensure that that $L$ is closed below.
 The second parameter slice $L_k^2$ will be denoted hereafter by $L_k Z$.  
   
We note that each cluster $S$ in $L_k Z(s)$ for some $s \in \RR_{\geq 0}$ is the path component of some $z \in \big(L_{s,k}(Z)\big)_0 \subseteq Z$.  For any $(t, T) \in \Gamma(Z; L_k Z)$ with $(s, S) \leq (t, T)$, $T$ is the path component of $z$ in $L_{t,k}(Z)\supseteq L_{s,k}(Z)$.  We put $S= [z]_s$ and $T= [z]_t$, unless the location of $[z]$ is clear from context. 

As in \cite{DHT}, we use the following notation for the $n^{th}$ unordered configuration space of $Y$.
\begin{defn}
For any positive integer $n$, let $Y_{dis}^n$ be the set of all subsets of $Y$ consisting of exactly $n$ distinct elements. 
\end{defn}
Using $Y_{dis}^{k+1}$, we can describe the density of $Y$ around the points of $X$.
\begin{defn}\label{defn:nbs}
For any nonnegative integer, $k$, let 
\[ N_k(X, Y) = \sup_{x \in X} \  \inf_{S \in Y_{dis}^{k+1}} \ \sup_{y \in S} \ d(x,y). \]
\end{defn}
Then $N_k(X,Y) \leq p$ if each $x \in X$ has at least $k$ distinct $p$-neighbours in $Y$.
If $N_k(X, Y) \leq p$ then $N_\ell (X, Y) \leq q$ for all $\ell \leq k$ and $q \geq p$.

\begin{defn}
Let $L_0 X[c]$ be the hierarchical clustering of $X$ given by
\[ L_0 X[c] (s) = \begin{cases}
\pi_0 L_{s,0}(X) & s \geq c \\ \emptyset & \text{otherwise}
\end{cases} = \begin{cases}
\pi_0 V_s (X) & s \geq c \\ \emptyset & \text{otherwise.}
\end{cases}\]
\end{defn}
For any $c \geq 0$, the poset, $\Gamma(L_0 X[c])$ is a subset of $\Gamma(\L_0 X)$, although $\Lambda(L_0 X[c])$ is not necessarily a subset of $\Lambda(L_0 X)$.  All layer points of $L_0 X = \pi_0 V_\bullet (X)$ are branch points and $L_0 X[c]$ inherits this property.

This truncated clustering can be used, under certain conditions, to approximate the clustering $L_k Y$.
\begin{lem}\label{lem:posets} Suppose $c \geq 0$. Then $(X, L_0 X[c])$ is an $(\epsilon, \delta)$-approximation for $(Y, L_k Y)$ for all $\epsilon, \delta \geq 0$ satisfying
 \[N_k(X, Y) - \epsilon \leq c \leq \delta,\] and \[\delta \geq 2h(Y,X).\]
\end{lem}
\begin{proof}
Let $s \geq c$. Since \[N_k(X, Y) \leq c+\epsilon \leq s+\epsilon,\] each $x \in X$ has at least $k$ $(s+\epsilon)$-neighbours in $Y$.  So, for any $[x]_s$ in $L_0 X[c](s)$, the cluster $[x]_{s+\epsilon}$ exists in $L_k Y(s+\epsilon)$. Since $s+\epsilon \geq s$, $[x]_{s+\epsilon}$ must contain $i([x]_s)$.

Fix any map $\theta:Y \to X$ with $d\big(y, \theta(y)\big) \leq h(Y,X)$ for all $y \in Y$.  For convenience, we will assume $\theta(x)=x$ for all $x \in X$, but the result holds even if this assumption is dropped. For any $[y]_s \in L_k Y(s)$ and any $y_0 \in [y]_{s}$, there exists a sequence of points in $Y$,
\[ y_0, y_1, y_2, \ldots, y_{n-1}, y_n=y \]
satisfying $d(y_i, y_{i+1}) \leq s$ for all $0 \leq i < n$.  For any such $i$,
\begin{align*} d\big(\theta(y_i), \theta(y_{i+1}) \big) &\leq d(\theta(y_i), y_i) + d(y_i, y_{i+1}) + d(y_{i+1}, \theta(y_{i+1})) \\ &\leq s+2h(Y,X) \\ &\leq s+\delta. \end{align*}
So, $\theta([y]_s)$ is contained in the cluster, $[\theta(y)]_{s+\delta}$ in $L_k X(s+\delta)$.  

Thus $i_{\epsilon}$ and $\theta_\delta$ are well defined in the diagrams
\[\adjustbox{scale=0.9}{\begin{tikzcd}[column sep=tiny, scale=0.5, every label/.append style = {font = \normalsize}]
L_0 X[c](s) \ar[rr, "(id_Y)_{\scaleto{{\epsilon}+\delta}{7pt}}"] \ar[dr, "i_{\scaleto{\epsilon}{5pt}}", swap] 
& & L_0 X[c](s+\epsilon+\delta)   \\
& L_k Y(s+\epsilon) \ar[ur,"\theta_{\scaleto{\delta}{6pt}}", swap]  &  
\end{tikzcd}  \begin{tikzcd}[column sep=tiny, scale=0.5, every label/.append style = {font = \normalsize}]
 & L_0 X[c](s+\delta) \ar[dr, "i_{\scaleto{\epsilon}{5pt}}"] & \\
 L_k Y(s) \ar[ur,"\theta_{\scaleto{\delta}{6pt}}"] \ar[rr, swap,  "(id_Y)_{\scaleto{\epsilon+\delta}{7pt}}"] & & L_k Y(s+\epsilon+\delta).
\end{tikzcd}}\]  
for all $s$.
It remains to show the diagrams commute. 

For any $[y]_{s} \in L_k Y(s)$, 
\[ (id_Y)_{\epsilon+\delta}\left([y]_{s+\epsilon}\right) = [y]_{s+2\epsilon+\delta} = [\theta(y)]_{s+2\epsilon+\delta} = i_\epsilon \theta_\delta [y]_{s+\epsilon},  \]
since $d(y,\theta(y)) \leq h(Y,X) \leq \frac{1}{2} \delta \leq s+\epsilon+\delta$.
Likewise,
\[ (id_Y)_{\scaleto{\epsilon+\delta}{6pt}} \left([x]_s \right)= [x]_{s+\epsilon+\delta} = [\theta(x)]_{s+\epsilon+\delta} = \theta_\delta i_\epsilon [x]_s  \]
for all $[x]_s \in L_0 X [c](s)$.
\end{proof}

There is therefore a homotopy commutative diagram
\[\begin{tikzcd}[column sep=large]
    \Lambda(L_0 X[c]) \ar[d, "i_{\scaleto{\epsilon}{3.5pt}}"] \ar[r, "(id_Y)_{\scaleto{\epsilon+\delta}{6pt}}"]& \Lambda(L_0 X[c]) \ar[d, "i_{\scaleto{\epsilon}{3.5pt}}"]\\
\Lambda(L_k Y) \ar[r, "(id_Y)_{\scaleto{\epsilon +\delta}{6pt}}", swap] \ar[ur, "\theta_{\scaleto{\delta}{5pt}}"] & \Lambda(L_k Y).
\end{tikzcd}\] 

\begin{cor}\label{cor:X=Y}
For all ${\epsilon}, {\delta} \geq 0$ satisfying $N_k (X, X) - \epsilon \leq c \leq \delta$, $(X, L_0 X [c])$ is an $(\epsilon, \delta)$-approximation for $(X, L_k X)$.  In particular, $(X, L_0 X [c])$ is an $\left(N_k(X,X), 0\right)$-approximation for $(X, L_k X)$.
\end{cor}
This follows from Lemma~\ref{lem:posets} by taking $Y=X$, which implies $h(Y,X)=0$.  Note that $N_k(X,X)$ gives the greatest lower bound on the distance parameter, $s$, for which the vertex set of $L_{k, s} X$ is equal to $X$.

\subsection{Retracts of Lesnick Clusterings}
  As in \cite{Layers}, we call the distances between elements of $X$ the phase change numbers of $X$.  Since $X$ is finite, there are a finite number of phase change numbers,  
\[ s_0 < s_1 < s_2 < \cdots < s_U. \]
Fix some number $M$ such that $\pi_0 L_{s, 0}(X) = \{X\}$ for all $s \geq s_{M}$.  It always works to take $M = U$, but choosing a smaller value of $M$, if possible, will give a less restrictive condition on $X$ in Theorem~\ref{theorem:main}.

The distance parameter associated to any layer point of $(X ,L_0 X[c])$ must be a phase change number and must be no larger than $s_M$.  Using this fact, we prove the following theorem. 
\begin{theorem}\label{theorem:main}
Suppose $c, \epsilon, \delta$ are nonnegative numbers satisfying $\delta \geq 2h(Y,X)$ and $N_k(X,Y) - \epsilon \leq c \leq \delta$.  If \[c+\epsilon+\delta< s_{i+1}-s_i\]  for all $0 \leq i < M$, then $\Lambda(L_0 X[c])$
is a retract of $\Lambda(L_k Y)$. 
\end{theorem}

In fact, the condition $c+\epsilon+\delta < s_{i+1}-s_i$ can be replaced with the weaker condition that $s - t > \epsilon + \delta$ for all $(t, T) < (s, S) \in \Lambda(L_0 X[c])$.  However, checking the condition $s_{i+1} - s_i > c+\epsilon+\delta$ does not require computing the set of layer points.  

To prove Theorem~\ref{theorem:main}, we will show that in the diagram
\[\begin{tikzcd}[column sep=large]
    \Lambda(L_0 X[c]) \ar[d, "i_{\scaleto{\epsilon}{3.5pt}}"] \ar[r, "(id_Y)_{\scaleto{\epsilon+\delta}{6pt}}"]& \Lambda(L_0 X[c]) \ar[d, "i_{\scaleto{\epsilon}{3.5pt}}"]\\
\Lambda(L_k Y) \ar[r, "(id_Y)_{\scaleto{\epsilon +\delta}{6pt}}", swap] \ar[ur, "\theta_{\scaleto{\delta}{5pt}}"] & \Lambda(L_k Y).
\end{tikzcd}\]
the composition $\theta_\delta i_\epsilon : \Lambda(L_0 X[c]) \to \Lambda(L_0 X[c])$ is the identity.  We begin by bounding the shift of parameter in the map $i_\epsilon$.

\begin{lem}\label{lem:param}
Let $(t, [x]):= i_\epsilon (s, [x])$ for some $(s, [x]) \in \Lambda(L_0 X[c])$ with $s>c$.  Then \[s-2h(Y,X)\leq t \leq s+\epsilon.\]
\end{lem}
\begin{proof}
Since $(t, [x])$ is a layer point below $(s+\epsilon,[x]) \in \Gamma(L_k Y)$, $t \leq s +\epsilon$.

To prove $s - 2h(Y,X) \leq t$, we use the fact that the layer points of $L_0 X = \pi_0 V_\bullet (X)$ are all branch points.  This is true for the truncated clustering as well, so for any layer point $(s, [x]) \in \Lambda(L_0 X[c])$, either $s=c$ or, for some $x_0, x_1 \in [x]_s$,  $[x_0]_{w} \neq [x_1]_{w}$ for all $w < s$. 

We assume $s>c$, so there must exist $x_0, x_1 \in [x]_s$ satisfying the second condition.  Since $[x_0] =  [x_1]$ in $L_0 X[c](s) =\pi_0 L_{s,0}(X)=\pi_0 V_s (X)$ and $X \subseteq Y$, $[x_0] = [x_1]$ in $L_k Y(s) = \pi_0 L_{s,k}(Y)$ as well.  

If $(t, [x_0]):= m(s, [x_0])$ in $\Lambda(L_k Y)$, then $|[x_0]_t|=|[x_0]_s|$.  Hence, $[x_0]=[x]=[x_1]$ in $L_k Y(t)$.  
Then there exists a sequence of points in $Y$,
\[ x_0=y_0, y_1, \ldots, y_n = x_1 \]
where each $y_i$ has at least $k$ $t$-neighbours and $d(y_i, y_{i+1})\leq t$.
Thus,
\[ d\big(\theta(y_i), \theta(y_{i+1})\big) \leq d\big(\theta(y_i), y_i\big) + d(y_i, y_{i+1}) + d\big( y_{i+1}, \theta(y_{i+1})\big) \leq t+2h(Y,X) \]
for all $i$. Since $\theta(y_0)=x_0$ and $\theta(y_n)=x_n$, this implies $x_0$ and $x_1$ are in the same path component 
in $L_0 X[c]\big(t+2h(Y,X)\big)$.  Thus, $t+2h(Y,X) \geq s$.\end{proof}

This bound allows us to prove the main theorem.
 
\begin{proof}[Proof of Theorem~\ref{theorem:main}]
Suppose $c, \epsilon, \delta$ are nonnegative numbers satisfying $\delta \geq 2h(Y,X)$, $N_k(X,Y) - \epsilon \leq c \leq \delta$, and \[c+\epsilon+\delta< s_{i+1}-s_i\]  for all $0 \leq i < M$.
Consider the homotopy commutative diagram
\[\begin{tikzcd}[column sep=large]
    \Lambda(L_0 X[c]) \ar[d, "i_{\scaleto{\epsilon}{3.5pt}}"] \ar[r, "(id_Y)_{\scaleto{\epsilon+\delta}{6pt}}"]& \Lambda(L_0 X[c]) \ar[d, "i_{\scaleto{\epsilon}{3.5pt}}"]\\
\Lambda(L_k Y) \ar[r, "(id_Y)_{\scaleto{\epsilon +\delta}{6pt}}", swap] \ar[ur, "\theta_{\scaleto{\delta}{5pt}}"] & \Lambda(L_k Y).
\end{tikzcd}\]
If $(s, [x]) \in \Lambda(L_0 X[c])$, then $s=c$ or $s$ is a phase change number of $X$.
The first phase change number, $s_0$, is zero, so, \[s_1 = s_1-s_0 > c+\epsilon+\delta.\]

In the case where $s$ is a phase change number, we have for $(q, [x]) :=  \theta_\delta i_\epsilon (s, [x])$, \[q \leq c+\epsilon+\delta.\]  Then, 
\[ s_0=0 < c \leq q \leq c+\epsilon+\delta < s_1. \]
Since $q$ is therefore not a phase change number, $q = c$.

In the case where $s>c$, the distance from $s$ to the next phase change number is greater than $c+\epsilon+\delta$. 
 Lemma~\ref{lem:param} implies that, for $(t, [x]):= i_\epsilon (s,[x])$, \[s-2h(Y,X) \leq t  \leq s+\epsilon. \]
Since $2h(Y,X) \leq \delta$,
\[ s \leq s+\delta-2h(Y,X) \leq t+\delta \leq s+\epsilon+\delta.  \]
Because $(q, [x]):= \theta_\delta i_\epsilon(s, [x])$ is the maximal branch point below $(t+\delta, [x])$ and $(s, [x])$ is a branch point, \[s \leq q \leq s+\epsilon+\delta.\]
The distance from $s$ to the next phase change number is greater than $\epsilon+\delta$, hence $q=s$.
\end{proof}
Under the assumptions of Proposition 15, \[(id_Y)_{\epsilon+\delta}: \Lambda(L_0 X[c]) \to \Lambda(L_0 X[c])\] is the identity. So in this case, the upper triangle in the diagram of layer points commutes on the nose, not just up to homotopy.

\section{Special Cases of Layer Stability}
In this section we consider several special cases of Theorem~\ref{theorem:main}.  Suppose $\epsilon$, $\delta$, and $c$ are nonnegative numbers satisfying $N_k(X,Y)-\epsilon \leq c \leq \delta$ and $\delta \geq 2h(Y,X)$.  The map
\begin{align*}  i_\epsilon : \Lambda(L_0 X[c]) &\to \Lambda(L_k Y)
\end{align*}
exists as long as $\epsilon+c \geq N_k(X,Y)$.  So, in order to minimize the term $c+\epsilon+\delta$, we can start by recognizing that $\epsilon$ never needs to be greater than $N_k(X,Y)$ to define the interleaving of layer points.  For any fixed value $\epsilon$ between 0 and $N_k(X,Y)$, the minimal values for $c$ and $\delta$ are $N_k(X,Y)-\epsilon$ and $\max\{2h(Y,X), N_k(X,Y)-\epsilon\}$, respectively.
Choosing these values for $c$ and $\delta$ gives a homotopy commutative diagram for any $0 \leq \epsilon \leq N_k(X,Y)$,
\begin{align} \label{diagram}
\begin{tikzcd}[column sep=huge, ampersand replacement=\&]
    \Lambda(L_0 X[N_k(X,Y)-\epsilon]) \ar[d, "i_{\scaleto{\epsilon}{3.5pt}}"] \ar[r, "(id_Y)_{\scaleto{\epsilon+n}{6pt}}"] \& \Lambda(L_0 X[N_k(X,Y)-\epsilon]) \ar[d, "i_{\scaleto{\epsilon}{3.5pt}}"]\\
\Lambda(L_k Y) \ar[r, "(id_Y)_{\scaleto{\epsilon +n}{6pt}}", swap] \ar[ur, "\theta_{\scaleto{n}{3pt}}"] \& \Lambda(L_k Y)
\end{tikzcd}
\end{align}
where $n = \max\{ 2h(Y,X), N_k(X,Y)-\epsilon \}$.
If \[N_k(X,Y)+n=\max\{ N_k(X,Y)+2h(Y,X), 2N_k(X,Y)-\epsilon \}< s_{i+1}-s_i\] for all $0 \leq i < M$, then \[\theta_n i_\epsilon = (id_Y)_{n+\epsilon} = id_{\Lambda(L_0 X[N_k(X,Y)-\epsilon])}.\]   
Note that if $N_k(X,Y)+2h(Y,X) \geq s_{i+1}-s_i$ for some $i$, there is no choice of $\epsilon$ that will give direct commutativity in the upper triangle. However, if \[N_k(X,Y)+2h(Y,X) < s_{i+1}-s_i\] for all $i$ then any nonnegative $\epsilon$ between $N_k(X,Y)-2h(Y,X)$ and $N_k(X,Y)$ will guarantee direct commutativity. 

Setting $\epsilon=N_k(X,Y)$ in (\ref{diagram}) gives the following result. 
\begin{cor}\label{cor:smallparam}
If $N_k(X,Y)+2h(Y,X) < s_{i+1}-s_i$ for all $0 \leq i < M$, then there is a diagram,
\[\begin{tikzcd}
\Lambda(L_0 X) \ar[d] \ar[r, "id"]& \Lambda(L_0 X) \ar[d]\\
\Lambda(L_k Y) \ar[r] \ar[ur] & \Lambda(L_k Y)
\end{tikzcd}\]
where the bottom triangle commutes up to homotopy and the top triangle commutes directly, defining a retract
\[ \Lambda(L_k Y) \to \Lambda(L_0 X). \]   
\end{cor}

We can also consider the case where $X=Y$, implying $h(Y, X)=0$.  If $N_k(X, X) \leq p$, then for any nonnegative $c,\epsilon,$ and $ \delta$ satisfying $N_k(X,Y)-\epsilon \leq c \leq \delta$, Corollary~\ref{cor:X=Y} implies the existence of a homotopy commutative diagram of posets,
\[\begin{tikzcd}[column sep=large]
    \Lambda(L_0 X[c]) \ar[d, "i_{\scaleto{\epsilon}{3.5pt}}"] \ar[r, "(id_X)_{\scaleto{\epsilon+\delta}{6pt}}"]& \Lambda(L_0 X[c]) \ar[d, "i_{\scaleto{\epsilon}{3.5pt}}"]\\
\Lambda(L_k X) \ar[r, "(id_X)_{\scaleto{\epsilon +\delta}{6pt}}", swap] \ar[ur, "\theta_{\scaleto{\delta}{5pt}}"] & \Lambda(L_k X).
\end{tikzcd}\]

However, the condition
\[ c+\epsilon+\delta < s_{i+1}-s_i \]
for all $i$, which would guarantee $\theta_\delta$ is a retract,
can only be met for all $0 \leq i < M$ if $k=0$, and thus $N_k(X,X)=0$.  
\begin{note}
To see that $c+\epsilon+\delta < s_{i+1}-s_i$ for all $i$ implies $k=0$, suppose $k>0$ and therefore $N_k(X,X)>0$.
Take any $x \in X$ and $z \neq x$, an $N_k(X,Y)$-neighbour of $x$.  Then $d(x,z) \leq N_k(X,Y)$.  Since $d(x, z)$ is a positive phase change number, 
\[  s_1 - s_0 = s_1 - 0 \leq d(x, z) \leq N_k(X,Y). \]
However,
\[ N_k(X,Y) \leq N_k(X,Y)+\delta \leq c+\epsilon+\delta, \]
so, $s_1 - s_0 \leq c+\epsilon+\delta$.  
\end{note}

In the case where $k=0$, the only choice for $\epsilon$ in (\ref{diagram}) is $\epsilon=0=N_k(X,Y)$, which results in the rather unhelpful diagram
\[\begin{tikzcd}
\Lambda(L_0 X) \ar[d] \ar[r]& \Lambda(L_0 X) \ar[d, ]\\
\Lambda(L_0 X) \ar[r, ] \ar[ur, ] & \Lambda(L_0 X)
\end{tikzcd}\]
where all the maps are the identity.

A slightly more interesting example can be constructed for any $c$ which is less than $s_{i+1} - s_i$ for all $i$, by setting $\epsilon=0$ and $\delta=c$
\begin{cor}
For any nonnegative number, $c$, satisfying $c< s_{i+1}-s_i$ for all $0 \leq i < M$, there exists a commutative diagram,
\[\begin{tikzcd}\label{cor:X=Y_layers}
\Lambda(L_0 X[c]) \ar[d, "i_0"] \ar[r, "id"]& \Lambda(L_0 X[c]) \ar[d, "i_0"]\\
\Lambda(L_0 X) \ar[r, "id"] \ar[ur, "\theta_c"] & \Lambda(L_0 X).
\end{tikzcd}\]
\end{cor}
In this case, both triangles commute directly, not just up to homotopy, so this gives an isomorphism $i_0: \Lambda(L_0 X[c]) \to \Lambda(L_0 X)$.
That is, truncating the poset $\Gamma(L_0 X)$ by a small amount does not change structure of the layer point poset.  
This is unsurprising, and one can directly show the stronger result that there is
an isomorphism $\Lambda(L_0 X[c]) \to \Lambda(L_0 X)$ for any $c < s_1$.
This provides an alternate argument for Corollary~\ref{cor:smallparam}, as in the interleaving diagram of Theorem~\ref{theorem:main} the poset $\Lambda(L_0 X[c])$ can be replaced by $\Lambda(L_0 X)$ when $c< s_{i+1}-s_i$ for all $0 \leq i < M$.

\nocite{*}
\printbibliography

\end{document}